\documentclass[reqno,11pt]{amsart}

\usepackage{amsthm,amsfonts, amssymb, amscd}
\usepackage{euscript}

\usepackage{tikz-cd}

\usepackage{accents}

\usepackage{hyperref}
\hypersetup{%
   linktoc=page
}

\let\oldtocsection=\tocsection \let\oldtocsubsection=\tocsubsection
\renewcommand{\tocsection}[2]{\hspace{0em}\oldtocsection{#1}{#2}}
\renewcommand{\tocsubsection}[2]{\hspace{2em}\oldtocsubsection{#1}{#2}}

\newtheorem{theorem}{Theorem}[section] 
\newtheorem{lemma}[theorem]{Lemma}
\newtheorem{corollary}[theorem]{Corollary} 
\newtheorem{proposition}[theorem]{Proposition}

\theoremstyle{remark} \newtheorem{remark}[theorem]{Remark}
\theoremstyle{remark} \newtheorem{example}[theorem]{Example}
\theoremstyle{definition}

 \theoremstyle{definition}

 \theoremstyle{definition}

\newtheorem{definition}[theorem]{Definition}

\numberwithin{equation}{section}

\newcommand{\R}{\mathbb{ R}} 
\newcommand{\Z}{\mathbb{ Z}} 
 
\newcommand{\Hom}{\operatorname{Hom}}

\newcommand{\spec}{\operatorname{Spec}}
\newcommand{\Char}{\operatorname{Char}}

 \newcommand{\fm}{{\mathfrak m}}

 \newcommand{\ft}{{\mathfrak t}}

\newcommand{\ga}{\alpha} \newcommand{\gb}{\beta}
\newcommand{\gd}{\delta} \newcommand{\gD}{\Delta}
 \newcommand{\gre}{\epsilon}
\renewcommand{\gg}{\gamma} 
 \newcommand{\gl}{\lambda}
 \newcommand{\gm}{\mu} \newcommand{\gn}{\nu}
 
\newcommand{\gP}{\Phi}  \newcommand{\gs}{\sigma}
 \newcommand{\gt}{\theta}
 \newcommand{\gz}{\zeta}

 \newcommand{\cf}{\mathcal{F}}
 \newcommand{\ch}{\mathcal{H}}

\newcommand{\co}{\mathcal{O}}
 
\newcommand{\crr}{\mathcal{R}}
\newcommand{\ct}{\mathcal{T}}

 \newcommand{\vq}{\mathbf{q}}
\newcommand{\vr}{\mathbf{r}} \newcommand{\vs}{\mathbf{s}}
\newcommand{\vt}{\mathbf{t}} \newcommand{\ve}{\mathbf{e}}

\newcommand{\wh}{\widehat}

\newcommand{\sym}{\operatorname{Sym}}

\newcommand{\pam}{\Phi_{\operatorname{amb}}}

\newcommand{\pta}{\Phi_{\tan}}

\newcommand{\pno}{\Phi_{\operatorname{nor}}}

 \newcommand{\cone}{\operatorname{Cone}}
 \newcommand{\zcone}{\operatorname{Cone}_{\Z}}

 \newcommand{\pull}{i_V^{*}[\co_Y]_V}

 \newcommand{\reu}{\widetilde{\chi}}
 \newcommand{\oreu}{\widetilde{\chi}^{\circ}}
 \newcommand{\bd}{\partial} 

\newcommand{\vamkb}{(U^-(x)\cap U)\cdot xB}
\newcommand{\vamkm}{(U_P^-(x)\cap U^-)}
\newcommand{\vamkmb}{(U^-(x)\cap U^-)}

\newcommand{\xw}{X^w}
\newcommand{\xwp}{X_P^w}

\newcommand{\kxw}{Y_x^w}
\newcommand{\kxwp}{Y_{x,P}^w}

\newcommand{\pullk}{i_{V}^{*}[\co_{\kxw}]_{V}}
\newcommand{\pullkk}{i_V^{*}[\co_{\kxw}]_V}
\newcommand{\pullb}{i_x^{*}[\co_{\xw}]_{G/B}}
\newcommand{\pully}{i_x^{*}[\co_{Y_x^w}]}

\newcommand{\ws}{_{w,\vs}}

\newcommand{\es}{\emptyset}

\newcommand{\Y}{Y}

\newcommand{\zring}[1]{\mathbb{Z}[e^{-\lambda}, \lambda\in {#1}]} 
\newcommand{\tsy}{{T_xY}}
\newcommand{\tcy}{{TC_xY}}
\newcommand{\fld}{k}

 \allowdisplaybreaks

\makeatletter
\@namedef{subjclassname@1991}{Subject}
\makeatother

\begin{document}
\parskip=4pt \baselineskip=14pt

\title[Tangent spaces to Schubert varieties]{Equivariant K-theory and
  Tangent spaces to Schubert varieties}
\subjclass{Primary 14M15; Secondary 05E14.  Keywords: K-theory, Schubert variety, flag variety, Grassmannian, cominuscule}


\author{William Graham} \address{ Department of Mathematics,
  University of Georgia, Boyd Graduate Studies Research Center,
  Athens, GA 30602 } \email{wag@uga.edu}

\author{Victor Kreiman} \address{ Department of Mathematics,
  University of Wisconsin - Parkside, Kenosha, WI 53141 }
\email{kreiman@uwp.edu}

\date{\today}

\begin{abstract}
  Tangent spaces to Schubert varieties of type $A$ were characterized
  by Lakshmibai and Seshadri \cite{LaSe:84}.  This result was extended
  to the other classical types by Lakshmibai \cite{Lak:95},
  \cite{Lak:00}, and \cite{Lak2:00}.  We give a uniform
  characterization of tangent spaces to Schubert varieties in
  cominuscule $G/P$. Our results extend beyond cominuscule $G/P$; they
  describe the tangent space to any Schubert variety in $G/B$ at a
  point $xB$, where $x$ is a cominuscule Weyl group element in the
  sense of Peterson.  Our results also give partial information about
  the tangent space to any Schubert variety at any point.  Our method
  is to describe the tangent spaces of Kazhdan-Lusztig varieties, and
  then recover results for Schubert varieties.  Our proof uses a
  relationship between weights of the tangent space of a variety with
  torus action, and factors of the class of the variety in torus
  equivariant $K$-theory. The proof relies on a formula for Schubert
  classes in equivariant $K$-theory due to Graham \cite{Gra:02} and
  Willems \cite{Wil:06}, as well as a theorem on subword complexes due
  to Knutson and Miller \cite{KnMi:04}, \cite{KnMi:05}.
\end{abstract}

\maketitle

\tableofcontents

\section{Introduction}\label{s.introduction}

One goal in the study of Schubert varieties is to understand their
singularities. A related goal is to understand their Zariski tangent
spaces, or equivalently, the weights of their Zariski tangent spaces
at fixed points of the action of a maximal torus.  A description of
these tangent spaces in type A was given by Lakshmibai and Seshadri
\cite{LaSe:84}. In \cite{Lak:95}, \cite{Lak:00}, and \cite{Lak2:00},
Lakshmibai extended this result to all classical types (see also
\cite[Chapter 5]{BiLa:00}). We give a different description of tangent
spaces to Schubert varieties, which is uniform across all types. Our
description, however, recovers only part of the tangent space, except
in certain cases, such as Schubert varieties in cominuscule $G/P$, in
which it recovers the entire tangent space. Our results hold for an
algebraically closed ground field $k$ of characteristic 0.  The
results of \cite{LaSe:84}, \cite{Lak:95}, \cite{Lak:00}, and
\cite{Lak2:00} hold in arbitrary characteristic.

Rather than studying Schubert varieties directly, we focus on the
smaller Kazhdan-Lusztig varieties, which differ locally only by a
well-prescribed affine space.  We study general Kazhdan-Lusztig
varieties, but we do not attempt to recover all weights of the tangent
space. Rather, we restrict our attention to those weights of the
tangent space which are integrally indecomposable in an ambient space
$V$, which is to say that they cannot be written as the sum of other
weights of $V$. We characterize such weights. When all weights of $V$
are integrally indecomposable in $V$, our characterization captures
all weights of the tangent space. This occurs, for example, for
Kazhdan-Lusztig varieties in cominuscule $G/P$, or more generally, for
any Kazhdan-Lusztig variety at a $T$-fixed point (i.e., point of
tangency) which is a cominuscule Weyl group element.

\subsection{Statement of results}\label{ss.statement}

Let $G$ be a semisimple algebraic group defined over an algebraically
closed field $k$ of characteristic 0. Let $P\supseteq B\supseteq T$ be
a parabolic subgroup, Borel subgroup, and maximal torus of $G$
respectively. We denote the set of weights of a representation $E$ of
$T$ by $\gP(E)$.  Let $W$ be the Weyl group of $(G,T)$, and $S$ the
set of simple reflections in $W$ relative to $B$.

Fix $w\leq x\in W$. Let $\xw$ be the Schubert variety
$\overline{B^-wB}$, and $\kxw$ the Kazhdan-Lusztig variety $BxB\cap
\overline{B^-wB}$, in $G/B$. The Kazhdan-Lusztig variety $\kxw$ (and
thus its tangent space at $x$, $T_x\kxw$) is an affine subvariety of
an ambient space $V$ in $G/B$ with weights $\gP(V)=I(x^{-1})$, the
inversion set of $x^{-1}$. If $\vs=(s_1,\ldots,s_l)$, $s_i\in S$, is a
reduced expression for $x$, then the elements of $I(x^{-1})$ are given
explicitly by the formula $\gg_i=s_1\cdots s_{i-1}(\ga_i)$,
$i=1,\ldots, l$, where $\ga_i$ is the simple root corresponding to
$s_i$.

Our main result is the following theorem (see Theorem
\ref{t.summary}):

\noindent \textbf{Theorem A}. {\it Suppose $\gg_j$ is integrally
  indecomposable in $I(x^{-1})$. Then the following are equivalent:
\begin{enumerate}
\item $\gg_j\in \gP(T_xY_x^w)$.
\item There exists a reduced subexpression of
  $(s_1,\ldots,\wh{s}_j,\ldots,s_l)$ for $w$.
\item The Demazure product of $(s_1,\ldots,\wh{s}_j,\ldots,s_l)$ is
  greater than or equal to $w$.
\end{enumerate}
} 

This theorem, which applies to Kazhdan-Lusztig varieties in $G/B$,
extends to Schubert varieties and to $G/P$. Moreover, when $x$ is a
cominuscule Weyl group element of $W$, all $\gg_j$ are integrally
indecomposable in $I(x^{-1})$, so Theorem A recovers all weights of
the tangent space.

\begin{remark}\label{r.not-int-dec}
  If $\gg_j$ is not integrally indecomposable in $I(x^{-1})$, then
  (ii) and (iii) of Theorem A are still equivalent, but (i) is no
  longer equivalent to (ii) and (iii) in general.
\end{remark}

\begin{remark}\label{r.tinv-cur}
  Let us denote by $TE_xY_x^w$ the span of the tangent lines to
  $T$-invariant curves through $x$ in $Y_x^w$; then
  $TE_xY_x^w\subseteq T_x Y_x^w$. It is known that condition (iii) of
  Theorem A, with the Demazure product of
  $(s_1,\ldots,\wh{s}_j,\ldots,s_l)$ replaced by the ordinary product
  $s_1\cdots\wh{s}_j\cdots s_l$, gives a characterization of all
  weights of $TE_xY_x^w$ (and not just the integrally indecomposable
  weights) \cite{Car:95} \cite{CaKu:03}. Thus, Theorem A can be viewed
  as a 
  characterization of the integrally indecomposable weights of $T_x
  Y_x^w$ which is similar to this known characterization of all
  weights of the smaller space $TE_x Y_x^w$.
\end{remark}

\begin{remark}
  The paper \cite{GrKr:21} proves that in simply-laced types, the
  Demazure product of $(s_1,\ldots,\wh{s}_j,\ldots,s_l)$ of Theorem A
  (iii) is equal to the ordinary product $s_1\cdots \wh{s}_j\cdots
  s_l$, provided that $\gg_j$ is integrally indecomposable in
  $I(x^{-1})$. As a corollary, it is proved that in simply-laced
  types, when $x$ is a cominuscule Weyl group element, $\gP(T_x\xw) =
  \gP(TE_x\xw)$.
\end{remark}

\subsection{Outline of proof}

Our proof of Theorem A uses equivariant $K$-theory. Let us fix
notation and give some basic definitions and properties. If $T$ acts
on a smooth scheme $M$, the Grothendieck group of $T$-equivariant
coherent sheaves (or vector bundles) on $M$ is denoted by $K_T(M)$. If
$N$ is a $T$-stable subscheme of $M$, then the class in $K_T(M)$ of
the pushfoward of the structure sheaf $\co_N$ of $N$ is denoted by
$[\co_N]_M$, or sometimes just $[\co_N]$.  A $T$-equivariant vector
bundle on a point is a representation of $T$, so
$K_T(\{\text{point}\})$ can be identified with $R(T)$, the
representation ring of $T$. The inclusion $i_m:\{m\} \to M$ of a
$T$-fixed point induces a pullback $i_m^{*}:K_T(M)\to
K_T(\{m\})=R(T)$.

Consider for the moment a more general situation than that of the
previous subsection: $V$ a representation of $T$ such that all weights
of $V$ lie in an open half-space and have multiplicity one,
$\Y\subseteq V$ a $T$-stable subscheme, and $x\in \Y$ a $T$-fixed
point. The structure sheaf $\co_{\Y}$ defines a class $[\co_\Y]\in
K_T(V)$. We show that the factors of $i_x^{*}[\co_Y]\in R(T)$ contain
information about the tangent space $T_xY$.  Let us say that
$1-e^{-\gt}$ is a {\it simple factor} of $i_x^{*}[\co_\Y]$ if
$i_x^{*}[\co_\Y]=(1-e^{-\gt})Q$ for some $Q\in R(T)$ which is a
polynomial in $e^{-\gl}$, $\gl\in \gP(V)\setminus\{\gt\}$. We prove
(see Proposition \ref{p.ind-remov}) \vspace{.5em}

\noindent \textbf{Proposition B}. {\it Suppose $\gt$ is integrally
  indecomposable in $\gP(V)$. Then $\gt\in \gP(T_x\Y)$ if and only if
  $1-e^{-\gt}$ is not a simple factor of $i_x^{*}[\co_\Y]$.}
\vspace{.5em}

Now set $V$ and $x$ as in Subsection \ref{ss.statement} and set $Y$ to
be the Kazhdan-Lusztig variety $\kxw$.  For $\gt\in\gP(V)=I(x^{-1})$,
we have $\gt=\gg_j$ for some $j$. Proposition B then becomes
\vspace{.5em}

\noindent \textbf{Proposition C}. {\it Suppose $\gg_j$ is integrally
  indecomposable in $I(x^{-1})$. Then $\gg_j\in \gP(T_x\Y_x^w)$ if and
  only if $1-e^{-\gg_j}$ is not a simple factor of
  $i_x^{*}[\co_{\Y_x^w}]$.} \vspace{.5em}

This characterization of $\gP(T_x\kxw)$ would appear to suffer from a
computational difficulty: determining whether $1-e^{-\gg_j}$ is a
factor of $i_x^{*}[\co_{Y_x^w}]$, let alone whether it is a simple
factor, seems to be a nontrivial problem. It requires some sort of
division algorithm in $R(T)$. One approach would be to search for an
expression for $\pully$ as a sum of terms in which $1-e^{-\gg_j}$
appears explicitly as a factor of each summand. To rule out the
possibility that $1-e^{-\gg_j}$ is a factor of $i_x^{*}[\co_{Y_x^w}]$,
then, one would need to show that no such expression exists. This
would presumably require knowledge of all possible expressions for
$\pully$.

We show a way around this apparent computational difficulty.  When
$\gg_j$ is integrally indecomposable in $I(x^{-1})$, the question of
whether or not $1-e^{-\gg_j}$ is a simple factor of
$i_x^{*}[\co_{Y_x^w}]$ can be answered by using a single expression
for $\pully$ due to Graham \cite{Gra:02} and Willems \cite{Wil:06}:
\begin{equation}\label{e.gw-intro}
  i_x^{*}[\co_{Y_x^w}] = 
  \sum_{\vt\in \ct_{w,\vs}}(-1)^{e(\vt)}\prod_{i\in \vt}(1-e^{-\gg_i})\in R(T),
\end{equation}
where $\ct_{w,\vs}$ is the set of sequences $\vt=(i_1,\ldots,i_m)$,
$1\leq i_1<\cdots<i_m\leq l$, such that $H_{s_{i_1}}\cdots
H_{s_{i_m}}=H_w$ in the 0-Hecke algebra, and $e(\vt)=m-\ell(w)$.  More
precisely, we prove  (see Theorem \ref{t.sk-standard}): \vspace{.5em}

\noindent \textbf{Theorem D}.  {\it Suppose $\gg_j$ is integrally
  indecomposable in $I(x^{-1})$. Then $1-e^{-\gg_j}$ is a simple
  factor of $\pully$ if and only if $1-e^{-\gg_j}$ occurs explicitly
  as a factor of every summand of $\sum_{\vt\in
    \ct_{w,\vs}}(-1)^{e(\vt)}\prod_{i\in \vt}(1-e^{-\gg_i})$, i.e., if
  and only if $j$ belongs to every $\vt\in \ct_{w,\vs}$.  }
\vspace{.5em}

We note that one direction of this theorem follows immediately from
\eqref{e.gw-intro}. Combining Proposition C and Theorem D yields
\vspace{.5em}

\noindent \textbf{Theorem E}. {\it Suppose that $\gg_j$ is integrally
  indecomposable in $I(x^{-1})$. Then $\gg_j\in \gP(T_xY_x^w)$ if and
  only if $j$ does not belong to every $\vt\in \ct_{w,\vs}$.}
\vspace{.5em}

Now the equivalence of (i) and (ii) of Theorem A is essentially a
reformulation of Theorem E, using some properties of 0-Hecke
algebras. The equivalence of (ii) and (iii) is due to Knutson-Miller
\cite[Lemma 3.4 (1)]{KnMi:04}.

The paper is organized as follows. In Section \ref{s.preliminaries},
we recall definitions and properties of equivariant $K$-theory and
weights of tangent spaces to schemes with $T$-actions.  In
Section \ref{s.removable} we prove Proposition B. In Section
\ref{s.euler}, we give a corollary to a theorem by Knutson-Miller on
subword complexes \cite{KnMi:04}, \cite{KnMi:05}.  Our proof of
Theorem D relies on this corollary.  In Section \ref{s.particular}, we
apply the material of the previous sections to Kazhdan-Lusztig
varieties in order to prove Proposition C and Theorems D and E. In
Section \ref{s.partial}, we show how to extend these results to $G/P$
and discuss the case of cominuscule Weyl group elements and
cominuscule $G/P$.

The related paper \cite{GrKr:21} examines {\it rationally
  indecomposable} weights of the ambient space $V$. Rational
indecomposability is a stricter condition than integral
indecomposability, so the set of rationally indecomposable weights is
contained in the set of integrally indecomposable weights. For this
smaller set of weights, \cite{GrKr:21} obtains stronger results.  For
example, it is shown that the elements of $\gP(T_xY_x^w)$ which are
rationally indecomposable in $I(x^{-1})$ lie in $\gP(TE_xY_x^w)$.
Several results of \cite{GrKr:21} rely on those of this paper.

\section{Preliminaries}\label{s.preliminaries}

Let $\fld$ be an algebraically closed field of characteristic 0.  We
work in the category of schemes over $\fld$. A point on a scheme will
always refer to a closed point. If $W$ is a finite dimensional vector
space over $\fld$, then one can give a bijection between the vectors
of $W$ and the (closed) points of the affine scheme
$\spec(\sym(W^{*}))$ \cite[Corollary 1.11]{GoWe:10}. As is customary,
we will often identify $W$ with this scheme, and in this context refer
to $\sym(W^{*})$ as the coordinate ring of $W$.

In this section, we collect information concerning equivariant
$K$-theory, tangent spaces, and tangent cones. We include proofs for
the convenience of the reader.

\subsection{The pullback to a fixed point in $T$-equivariant $K$-theory}

Let $T=(\fld^{*})^n$ be a torus, and let $\wh{T}=\Hom(T,\fld^{*})$ be the
character group of $T$. The mapping $\gl\mapsto d\gl$ from a character
to its differential at $1\in T$ embeds $\wh{T}$ in the dual $\ft^{*}$
of the Lie algebra of $T$. We will usually view $\wh{T}$ as a subset
of $\ft^{*}$ under this embedding and express the group operation
additively. If $\gl$ denotes an element of $\wh{T}$ viewed as an
element of $\ft^{*}$, then the corresponding homomorphism $T\to
\fld^{*}$ is written as $e^{\gl}$. The representation ring $R(T)$ is
defined to be the free $\Z$-module with basis $e^{\gl}$, $\gl\in
\wh{T}$, with multiplication given by $e^{\gl}e^{\gm}=e^{\gl+\gm}$.

Let $V$ be a finite dimensional representation of $T$ such that all
weights of $V$ have multiplicity one and lie in an open half-space in
the real span of the characters of $T$.  Denote the set of weights of
$T$ on $V$ by $\Phi(V)$ and the set of nonnegative integer linear
combinations of elements of $\Phi(V)$ in $\ft^{*}$ by
$\zcone\Phi(V)$. The dual representation $V^{*}$ has weights
$-\Phi(V)$, and the coordinate ring $\fld[V]=\sym(V^{*})$ of $V$ has
weights $-\cone_{\Z}\Phi(V)$. Denote
$\prod_{\gl\in\Phi(V)}(1-e^{-\gl})$ by $\gl_{-1}(V^{*})$. For
$\gP_A\subseteq \gP(V)$, let $\zring{\gP_A}$ denote the subring of
$R(T)$ generated over $\Z$ by $e^{-\gl}$, $\gl\in \gP_A$.

\begin{lemma}\label{l.bcone}
  If $\gn \in \cone_{\Z}\Phi(V)$, then $e^{-\gn}$ can be expressed as
  a monomial in $e^{-\gl}$, $\gl\in \Phi(V)$.
\end{lemma}
\begin{proof}
  Write $\gn=c_1\gl_1+\cdots+c_t\gl_t$, where $\gl_i\in\Phi(V)$ and
  $c_i$ are nonnegative integers. Then $e^{-\gn} = e^{\sum -c_i\gl_i}
  = \prod (e^{-\gl_i})^{c_i}$.
\end{proof}

The map $i_0^{*}:K_T(V)\to R(T)$ is an isomorphism, which we denote by
$i_V^{*}$.

\begin{lemma}\label{l.cl-subvar}
  Let $\Y$ be a $T$-stable closed subscheme of $V$. Then
  $i_V^{*}[\co_\Y]_V\in \zring{\gP(V)}$.
\end{lemma}
\begin{proof} We adapt a method appearing in \cite{Ros:89}.  Denote
  $\fld[V]$ by $R$. All modules in this proof will be $T$-stable
  $R$-modules, and all maps $T$-equivariant $R$-homomorphisms.

Consider the projection
  \begin{equation}\label{e.rproj}
    R\to \fld[\Y] \to 0
  \end{equation}
  The kernel is a $T$-stable ideal $I$ of $R$ which is generated by a
  finite number of weight vectors $r_{1,1}, \ldots,r_{1,n_1}$. Let
  $\nu_{1,j}$ be the weight of $r_{1,j}$; $\fld_{\nu_{1,j}}$ the
  $T$-representation of weight $\nu_{1,j}$; and
  $F_{1}=\oplus_{j=0}^{n_1} R\otimes \fld_{\nu_{1,j}}$. Note that
  $R$ acts on the first factor of $F_1$ and $T$ on both, and that
  $\gP(F_1)\subseteq \gP(R)$.  There exists
  a map $f_{1}:F_{1}\to R$ such that
  \begin{equation}\label{e.rprojtw}
    F_1\to R\to \fld[\Y]\to 0
  \end{equation}
  is exact ($f_1$ maps $1\otimes 1$ from the $j$-th summand of $F_1$
  to $r_{1,j}$).

  The kernel of $f_1$ is finitely generated over $R$ (since $F_1$ is
  finitely generated and $R$ is Noetherian), and thus is generated by
  a finite number of weight vectors. Thus the above procedure can be
  repeated to produce a module $F_2$ and map $F_2\to F_1$, which, when
  appended to \eqref{e.rprojtw}, yields an exact sequence. Moreover,
  $\gP(F_2)\subseteq \gP(F_1)\subseteq \gP(R)$. When iterated, this
  procedure must terminate, by the Hilbert Syzygy Theorem. The
  resulting complex is a resolution of $\fld[\Y]$:
  \begin{equation}\label{e.rprojth}
    0 \to F_d\to \cdots \to F_1 \to R\to \fld[\Y] \to 0    
  \end{equation}
  where $F_i=\oplus_jR\otimes \fld_{\nu_{i,j}}$ and
  $\gP(F_i)\subseteq \gP(R)$. Thus $\nu_{i,j}\in\Phi(F_i)\subseteq
  \Phi(R)=-\cone_{\Z}\Phi(V)$.

  The resolution \eqref{e.rprojth} corresponds to a resolution of
  $\co_\Y$ over $\co_V$:
  \begin{equation*}
    0\to \cf_d \to \cdots \to \cf_1 \to \co_V \to \co_\Y \to 0
  \end{equation*}
  where $\cf_i=\oplus_j \co_V\otimes \fld_{\nu_{i,j}}$. Since
  $[\co_V]_V=1$, we have
  \begin{equation*}
    i_V^{*}[\co_\Y]_V=1+\sum_{i,j}(-1)^i e^{\nu_{i,j}} i_V^{*}[\co_V]_V = 
    1+\sum_{i,j}(-1)^ie^{\nu_{i,j}}
  \end{equation*}
  By Lemma \ref{l.bcone}, this lies in $\zring{\gP(V)}$.
\end{proof}

\subsection{Weights of tangent and normal spaces} 
The coordinate ring of $V$ is $\sym(V^{*})=\fld[x_{\gl}, \gl\in
\Phi(V)]$, a polynomial ring, where $x_{\gl}$ denotes a vector of
$V^{*}$ of weight $-\gl$.  Observe that this polynomial ring is graded
by the coordinates $x_{\gl}$ and also by the weights of the $T$
action.  The weights of any $T$-stable subspace of $V$ form a subset
of $\Phi(V)$, whose corresponding weight vectors span the subspace.
Thus there is a bijection between the $T$-stable subspaces of $V$ and
the subsets of $\Phi(V)$. The coordinate ring of a $T$-stable subspace
$Z$ is $\sym(Z^{*})=\fld[x_{\gl}, \gl\in \Phi(Z)]$.

Let $\Y\to V$ be a $T$-equivariant closed immersion, with $T$-fixed
point $x\in \Y$ mapping to $0\in V$.  Let $\fm$ be the maximal ideal
of the local ring of $Y$ at $x$.  The tangent space to $Y$ at $x$,
denoted $\tsy$, is defined to be $(\fm/\fm^2)^{*}$, a vector space
over $\fld$. It embeds naturally in the tangent space to $V$ at $0$,
which is isomorphic to $V$ \cite[(6.2), (6.3)]{GoWe:10}. 

Let
\begin{equation*}
  B=\sym(\fm/\fm^2), \qquad  C=\oplus_{i\geq 0}\fm^i/\fm^{i+1}.
\end{equation*}
Note that the degree one components of $B$ and $C$, which are denoted
by $B_1$ and $C_1$ respectively, are both equal to $\fm/\fm^2$.  We
will often identify $\tsy$ with the affine scheme $\spec(B)$. The
tangent cone to $Y$ at $x$, denoted $TC_xY$, is defined to be
$\spec(C)$.  The projection $B\twoheadrightarrow C$ induces an
inclusion $\tcy\hookrightarrow \tsy$.

Both $\tsy$ and $\tcy$ are $T$-stable, and $\tcy\hookrightarrow \tsy$
is $T$-equivariant. The coordinate ring $B$ of $\tsy$ is equal to
$\fld[x_{\gl}, \gl\in \Phi(\tsy)]$, with character
\begin{equation*}
  \Char B = \frac{1}{\prod_{\gl\in\Phi(\tsy)}(1-e^{-\gl})}
  = \frac{1}{\gl_{-1}((\tsy)^{*})}. 
\end{equation*}
This lives in $\wh{R}(T)$, the set of expressions of the form
$\sum_{\gl\in \wh{T}}c_{\gl}e^{\gl}$. Similarly, we have a formula for
the character of $C$.

\begin{proposition}\label{p.phivw}
  $\Char C=\dfrac{i_V^{*}[\co_{\tcy}]_V}{\gl_{-1}(V^{*})}=
  \dfrac{i_\tsy^{*}[\co_{\tcy}]_\tsy}{\gl_{-1}((\tsy)^{*})}$.
\end{proposition}
\begin{proof}
  The first equality is proved in \cite[Proposition 2.1]{GrKr:15} and
  the second in \cite[(3.10)]{GrKr:17}.
\end{proof}

\begin{proposition}\label{p.classcone}
  $[\co_{\tcy}]_V=[\co_\Y]_V$ in $K_T(V)$.
\end{proposition}
\begin{proof}
  See \cite[Proposition 3.1(2)]{GrKr:17}.
\end{proof}

\section{Factors in equivariant $K$-theory}\label{s.removable}

We keep the notations, definitions, and assumptions of the previous
section.  Denote $\Phi(V),\Phi(\tsy)$, and $\Phi(V/\tsy)=\Phi(V)\setminus
\Phi(\tsy)$ by $\pam,\pta$, and $\pno$ respectively. Then
$\pam=\pta\sqcup\pno$.  If $P\in \zring{\pam}$, then we will say that
$1-e^{-\gt}$ is a \textbf{simple factor} of $P$ if $P=(1-e^{-\gt})Q$,
for $Q\in \zring{\pam\setminus \{\gt\}}$.

\begin{example}\label{ex.remov}
  Suppose $\gl_1, \gl_2, \gl_3, \gl_4\in \pam$ are distinct,
  and $\gl_3=\gl_1+\gl_2$. Consider
  \begin{equation*}
    P=(1-e^{-\gl_1})(1-e^{-\gl_4}) +
    (1-e^{-\gl_2})(1-e^{-\gl_4})
    - (1-e^{-\gl_1})(1-e^{-\gl_2})(1-e^{-\gl_4}),
  \end{equation*}
  an element of $\zring{\pam}$.  Then $P$ can be expressed as
  $(1-e^{-(\gl_1+\gl_2)})(1-e^{-\gl_4})$. Thus both
  $1-e^{-(\gl_1+\gl_2)}$ and $1-e^{-\gl_4}$ are simple
  factors of $P$.
\end{example}

\begin{remark}\label{r.explicit}
  In Section \ref{ss.normal}, we will need to distinguish between the
  nature of the two factors $1-e^{-(\gl_1+\gl_2)}$ and
  $1-e^{-\gl_4}$ of $P$ in Example \ref{ex.remov}. While the latter
  factor appears explicitly as a factor of each summand, and thus is
  easily identifiable as a factor of $P$, the former does not. We will
  refer to $1-e^{-\gl_4}$ as an {\it explicit factor} of the
  expression $P$.

  There are usually many ways to express an element $P\in
  \zring{\pam}$. Explicit factors depend on the particular expression
  of $P$, while (non-explicit) factors do not.
\end{remark}

We wish to study whether it is possible to determine whether a weight
$\gt$ lies in $\pno$ or $\pta$ based on whether or not $1-e^{-\gt}$ is
a simple factor of $i_V^{*}[\co_\Y]_V$. We begin with the following
observation:

\begin{proposition}\label{p.nor-remov}
  If $\gt\in \pno$, then $1-e^{-\gt}$ is a simple factor of
  $i_V^{*}[\co_\Y]_V$.
\end{proposition}
\begin{proof}
  By Propositions \ref{p.classcone} and \ref{p.phivw},
 \begin{equation*}
   i_V^{*}[\co_\Y]_V=i_V^{*}[\co_{\tcy}]_V = 
   \frac{\gl_{-1}(V^{*})}{\gl_{-1}((\tsy)^{*})}i_\tsy^{*}[\co_{\tcy}]_\tsy.
 \end{equation*}
 Now,
 $\frac{\gl_{-1}(V^{*})}{\gl_{-1}((\tsy)^{*})}=\gl_{-1}((V/\tsy)^{*})=\prod_{\gl\in
   \pno}(1-e^{-\gl})$, and $1-e^{-\gt}$ occurs among the terms of this
 product exactly once. Moreover, since $\tcy$ is a closed algebraic
 subscheme of $\tsy$, Lemma \ref{l.cl-subvar} implies
 $i_\tsy^{*}[\co_{\tcy}]_\tsy\in \zring{\gP(\tsy)}\subseteq
 \zring{\pam\setminus\{\gt\}}$.
\end{proof}

The converse of this proposition is false, as illustrated by the
following example.

\begin{example}\label{ex.remov-tan}
  Suppose that $T$ acts on $V=\fld^3$, and that the standard basis
  vectors $\ve_1,\ve_2,\ve_3$ are weight vectors with corresponding
  weights $\gl_1,\gl_2,\gl_3=\gl_1+\gl_2$. Letting $x_1,x_2,x_3\in
  V^{*}$ denote the dual of the standard basis, we have, for $t\in T$,
  $t x_i = e^{-\gl_i}(t)x_i$, $i=1,2,3$. Let $\Y$ be the affine
  subscheme of $V$ defined by the ideal $I=(x_1x_2)$, and let $x$ be
  the origin. Then the ideal of the tangent space and tangent cone of
  $Y$ at $x$ are $\{0\}$ and $(x_1x_2)$ respectively. The tangent
  space of $\Y$ at $x$ is all of $V$, so $\pta=\pam$ and $\pno$ is
  empty. The tangent cone of $\Y$ at $x$ is the union of the
  $x_2x_3$-plane and the $x_1x_3$ plane, so its coordinate ring $C$
  has character
  \begin{equation*}
    \Char C =\left(\frac{1}{1-e^{-\gl_1}} +
      \frac{1}{1-e^{-\gl_2}}-1\right)\frac{1}{1-e^{-(\gl_1+\gl_2)}}
  \end{equation*}
   Additionally,
   $\gl_{-1}(V^{*})=(1-e^{-\gl_1})(1-e^{-\gl_2})(1-e^{-(\gl_1+\gl_2)})$,
   and thus by Proposition \ref{p.phivw},
 \begin{align*}
   i_V^{*}[\co_\Y]_V = i_V^{*}[\co_C]_V&=(1-e^{-\gl_1}) +
   (1-e^{-\gl_2}) -
   (1-e^{-\gl_1})(1-e^{-\gl_2})\\
   &= 1-e^{-(\gl_1+\gl_2)}
 \end{align*}
 Hence $1-e^{-\gl_3}=1-e^{-(\gl_1+\gl_2)}$ is a simple factor of
 $i_V^{*}[\co_\Y]_V$, but $\gl_3$ lies in $\pta$.
\end{example}

In this example, the fact that there exists $\gl_3$ in $\pta$ such
that $1-e^{-\gl_3}$ is a simple factor of $\pull$, thus violating the
converse of Proposition \ref{p.nor-remov}, appears to be related to
the fact that $\gl_3$ can be expressed as the sum of other weights of
$\pam$.  This suggests that the converse of Proposition
\ref{p.nor-remov} may hold if we restrict to weights $\gt$ which
cannot be expressed as such a sum.  This assertion is true, and is
proved in the following proposition.  Let us say that a weight of
$\pam$ is \textbf{integrally decomposable} if it can be expressed as a
positive integer linear combination of other elements of $\pam$, or
\textbf{integrally indecomposable} otherwise.

\begin{proposition}\label{p.ind-remov}
  Let $\gt$ be an integrally indecomposable element of $\pam$. Then
  the following are equivalent:
  \begin{enumerate}
  \item $\gt\in \pta$.
  \item $x_{\gt}\in B_1$.
  \item $x_{\gt}\in C_1$.
  \item $-\gt$ is a weight of $C$.
  \item $1-e^{-\gt}$ is not a simple factor of $i_V^{*}[\co_\Y]_V$.
  \end{enumerate}
\end{proposition}

\begin{proof}

  (v) $\Rightarrow$ (i) by Proposition \ref{p.nor-remov}; (i)
  $\Rightarrow$ (ii) since $B=\fld[x_{\gl},\gl\in\pta]$; (ii)
  $\Rightarrow$ (iii) since $C_1=\fm/\fm^2=B_1$; and (iii)
  $\Rightarrow$ (iv) because $x_{\gt}$ has weight $-\gt$.

  It remains to prove (iv) $\Rightarrow$ (v). We prove the
  contrapositive. Thus assume that $1-e^{-\gt}$ is a simple factor of
  $i_V^{*}[\co_\Y]_V$. Then $i_V^{*}[\co_\Y]_V=(1-e^{-\gt})Q$, where
  $Q\in \zring{\pam\setminus\{\gt\}}$. By Propositions \ref{p.phivw}
  and \ref{p.classcone},
  \begin{equation*}
    \Char C =\frac{Q}{\prod_{\gl\in \pam\setminus \{\gt\}}(1-e^{-\gl})} =
    Q\prod\nolimits_{\gl\in \pam\setminus \{\gt\}} (1+e^{-\gl}+e^{-2\gl}+\cdots).
  \end{equation*}
  Expanding, one obtains an infinite sum of terms $e^{-\gn}$, $\gn\in
  \zcone(\pam)$. None of these terms is equal to $e^{-\gt}$. (This is
  because none of the factors in the above product for $\Char C$
  contain a term $e^{-\gt}$; since $\gt$ is integrally indecomposable
  in $\pam$, the term $e^{-\gt}$ cannot be obtained by expanding the
  product.)  Thus $-\gt$ is not a weight of $C$, as required.
\end{proof}

\section{Euler characteristics of subword complexes}
\label{s.euler}

In Section \ref{s.particular}, we will apply the results of the
previous section to Schubert varieties and Kazhdan-Lusztig varieties.
One main tool for this purpose is Corollary \ref{c.euler}, whose proof
relies on a theorem by Knutson-Miller on subword complexes
\cite{KnMi:05}, \cite{KnMi:04}.

\subsection{The reduced Euler characteristic}
In this subsection we give a brief review of simplicial complexes and
their Euler characteristics.

Recall that an (abstract) simplicial complex on a finite set $A$ is a
set $\Delta$ of subsets of $A$, called faces, with the property that
if $F\in \gD$ and $G\subseteq F$ then $G\in \gD$. The dimension of a
face $F$ is $\# F-1$, and the dimension of $\gD$ is the maximum
dimension of a face. A maximal face of $\gD$ is called a facet.  Note
that $\gD=\es$ and $\gD=\{\es\}$ are distinct simplicial complexes,
called the void complex and irrelevant complex respectively.  If
$\gD\neq \es$, then $\es$ must be a face of $\gD$.

The reduced Euler characteristic of $\gD$ is defined to be $\reu(\gD)
= \sum_{F\in \gD} (-1)^{\dim F}$. If $\gD\neq \es$, so that $\es\in
\gD$, then $\es$ contributes a summand of $-1$ to $\reu(\gD)$. From
this we see, for example, that $\reu(\{\es\})=-1$, but $\reu(\es)=0$.

Suppose that $\gD\neq \es$ or $\{\es\}$.  Denoting the elements of $A$
by $x_1, \ldots, x_m$, the set $A$ can be embedded in $\R^m$ by
mapping $x_i$ to the $i$th standard basis vector of $\R^m$.  For any
face $F$ of $\Delta$, let $|F|$ be the convex hull of its vertices in
$\R^m$.  The geometric realization of $\Delta$ is then defined to be
$|\Delta| = \bigcup_{F\in\Delta}|F|$, a topological subspace of
$\R^m$. If $|\gD|$ is homeomorphic to a topological space $\Y$, then
$\gD$ is called a triangulation of $\Y$. In this case, the reduced
Euler characteristic of $\gD$ is equal to the topological reduced
Euler characteristic of $\Y$.  If $\Y$ is a manifold with boundary and
its boundary $\bd \Y$ is nonempty, then there exists a subcomplex of
$\gD$ which is a triangulation of $\bd \Y$ \cite[Proposition
5.4.4]{Mau:80}. This subcomplex is called the boundary of $\gD$ and
denoted by $\bd \gD$.

For $m\geq 0$, let $B^m=\{x\in \R^m: \|x\|\leq 1\}$ and $S^m=\{x\in
\R^{m+1}: \|x\| = 1\}$, the $m$-ball and $m$-sphere respectively. Both
can be triangulated. In the sequel, when we refer to an $m$-ball or
$m$-sphere or their notations, $m\geq 0$, we will mean a triangulation
of the object. When we refer to the sphere $S^{-1}$, we will mean the
irrelevant complex $\{\es\}$. With these conventions, for $m\geq 0$,
$\reu(B^m)=0$, $\reu(S^{m-1})=(-1)^{m-1}$, $\bd B^m=S^{m-1}$, and $\bd
S^{m-1}=\es$.  (Observe that $\bd B^0$ is the irrelevant complex, but
$\bd S^{m-1}$ is the void complex.) For $\gD=B^m$ or $S^m$, define
$\oreu(\gD)=\reu(\gD)-\reu(\bd \gD)$.

\begin{lemma}\label{l.interior}
  $\oreu(B^m)=\oreu(S^m)=(-1)^m$, for $m\geq 0$.
\end{lemma}
\begin{proof}
  The proof is a calculation: $\oreu(B^m)=0-(-1)^{m-1}=(-1)^m$, and
  $\oreu(S^m)=(-1)^m-0=(-1)^m$.
\end{proof}

\subsection{0-Hecke algebras}\label{ss.hecke}

Let $G$ be a semisimple algebraic group, $B$ a Borel subgroup, $B^-$
the opposite Borel subgroup, and $T=B\cap B^-$ a maximal torus. Let
$W=N_G(T)/T$, the Weyl group of $G$. Let $S$ be the set of simple
reflections of $W$ relative to $B$.  The 0-Hecke algebra $\ch$
associated to $(W,S)$ over a commutative ring $R$ is the associative
$R$-algebra generated by $H_u$, $u\in W$, and subject to the following
relations: $H_1$ is the identity element, and if $u\in W$ and $s\in
S$, then $H_uH_s=H_{us}$ if $\ell(us)>\ell(u)$ and $H_uH_s=H_u$ if
$\ell(us)<\ell(u)$. If $\vq=(q_1,\ldots,q_l)$ is any sequence of
elements of $S$, define the Demazure product $\gd(\vq)\in W$ by the
equation $H_{q_1}\cdots H_{q_l}=H_{\gd(\vq)}$.  Define $\ell(\vq)=l$
and $e(\vq)=\ell(\vq)-\ell(\gd(\vq))$.

\subsection{Subword Complexes}

Let $\vs=(s_1,\ldots,s_l)$ be a sequence of elements of $S$ and $w\in
W$. The subword complex $\Delta(\vs,w)$ is defined to be the set of
subsequences $\vr = (s_{i_1},\ldots,s_{i_t})$, $1\leq i_1<\cdots
<i_t\leq l$, whose complementary subsequence $\vs\setminus \vr$
contains a reduced expression for $w$. One checks that $\Delta(\vs,w)$
is a simplicial complex. Subword complexes were introduced in
\cite{KnMi:04}, \cite{KnMi:05}. We will require that $\vs$ contains a
reduced expression for $w$.

\begin{remark}
  The requirement that $\vs$ contains a reduced expression for $w$
  implies that the empty sequence $\es$ lies in $\gD(\vs,w)$. In
  particular, $\gD(\vs,w)$ is not the void complex. It is possible,
  however, for $\gD(\vs,w)$ to be the irrelevant complex. For example,
  this occurs when $\vs=(s_1)$, $w=s_1$.
\end{remark}

The following theorem is \cite[Theorem 3.7]{KnMi:04}:

\begin{theorem}\label{t.km}
  The subword complex $\Delta(\vs,w)$ is either a ball or sphere. A
  face $\vr$ is in the boundary of $\Delta(\vs,w)$ if and only if
  $\delta(\vs\setminus \vr)\neq w$.
\end{theorem}

\begin{corollary}\label{c.kmtwo} $\displaystyle
  \sum_{\vq\subseteq \vs \atop\gd(\vq)= w}(-1)^{e(\vq)}=1$.
\end{corollary}
\begin{proof}
  If $\vr\subseteq \vs$, then $\vr\in \gD(\vs,w)$ if and only if
  $\vs\setminus \vr$ contains a reduced expression for $w$ if and only
  if $\gd(\vs\setminus \vr)\geq w$ \cite[Lemma 3.4 (1)]{KnMi:04}.  The
  dimension of face $\vr$ is equal to $\ell(\vr)-1$. Thus
  $\sum_{\{\vr\subseteq \vs, \gd(\vs\setminus\vr)\geq
    w\}}(-1)^{\ell(\vr)-1}=\reu(\gD(\vs,w))$. By the second statement
  of Theorem \ref{t.km}, $\sum_{\{\vr\subseteq \vs,
    \gd(\vs\setminus\vr)> w\}}(-1)^{\ell(\vr)-1}=\reu( \bd
  \gD(\vs,w))$. Hence
  \begin{equation*}
    \sum_{\vr\subseteq \vs\atop \gd(\vs\setminus\vr)=
      w}(-1)^{\ell(\vr)-1}=\oreu(\gD(\vs,w)).
  \end{equation*}
  But since $\gD(\vs,w)$ is either a ball or sphere,
  $\oreu(\gD(\vs,w))=(-1)^{\dim \gD(\vs,w)}$, by Lemma
  \ref{l.interior}. To compute $\dim \gD(\vs,w)$, observe that if
  $\vr\in \gD(\vs,w)$ has maximal length, then $\vs\setminus \vr$ is a
  reduced word for $w$; thus $\ell(\vs)-\ell(\vr)=\ell(w)$, so
  $\ell(\vr)=\ell(\vs)-\ell(w)$. Thus $\dim
  \gD(\vs,w)=\ell(\vs)-\ell(w)-1$. We conclude
  \begin{equation*}
    \sum_{\vr\subseteq \vs\atop \gd(\vs\setminus\vr)=
      w}(-1)^{l(\vr)-1}=(-1)^{\ell(\vs)-\ell(w)-1}.
  \end{equation*}
  Multiplying both sides of this equation by
  $(-1)^{\ell(\vs)-\ell(w)-1}$, we obtain
  \begin{equation*}
  \sum_{\vr\subseteq \vs\atop
    \gd(\vs\setminus\vr)= w}(-1)^{e(\vs\setminus \vr)}=1,
\end{equation*}
since, if $\gd(\vs\setminus \vr)=w$, then we have
$(-1)^{\ell(\vs)-\ell(w)-1}(-1)^{\ell(\vr)-1}
=(-1)^{\ell(\vs)-\ell(\vr)-\ell(w)}=(-1)^{\ell(\vs\setminus
  \vr)-\ell(\gd(\vs\setminus \vr))}=(-1)^{e(\vs\setminus \vr)}$.  Now
the desired equation is obtained by re-indexing this summation. Rather
than summing over subsequences $\vr$ of $\vs$, one sums over their
complementary subsequences $\vq$.
\end{proof}

\begin{definition}\label{d.tws}
  Let $w\in W$ and let $\vs=(s_1,\ldots,s_p)$ be a sequence of simple
  reflections in $S$. Define $\ct_{w,\vs}$ to be the set of sequences
  $\vt=(i_1,\ldots,i_m)$, $1\leq i_1<\cdots<i_m\leq p$, such that
  $H_{s_{i_1}}\cdots H_{s_{i_m}}=H_w$. Then $\ell(\vt)=m$ and
  $e(\vt)=\ell(\vt)-\ell(w)$.
\end{definition}

\begin{corollary}\label{c.euler}
  $\sum_{\vt\in \ct_{w,\vs}}(-1)^{e(\vt)}=1$, if $\ct_{w,\vs}\neq \emptyset$. 
\end{corollary}
\begin{proof} 
  We have
  \begin{equation*}
    \sum_{\vt\in \ct_{w,\vs}}(-1)^{e(\vt)} = 
    \sum_{\vq\subseteq \vs \atop \delta(\vq)=w}(-1)^{e(\vq)}  = 1, 
  \end{equation*}
  where the first equality is obtained by re-indexing and the second
  equality is Corollary \ref{c.kmtwo}.  Note that the hypothesis
  $T_{w,\vs}\neq \emptyset$ assures us that $\vs$ contains a reduced
  expression for $w$, a requirement for Corollary \ref{c.kmtwo}.
\end{proof}

We remark that the expression $\sum_{\vt\in \ct_{w,\vs}}(-1)^{e(\vt)}$
appearing in Corollary \ref{c.euler} has elements in common with the
expression for $\pullb$ given by \eqref{e.gw}. The similarity between
these expressions is critical to our proof of Theorem
\ref{t.sk-standard}.

\section{Applications to Kazhdan-Lusztig and Schubert
  varieties}\label{s.particular}

In Section \ref{s.removable} we saw that the pullback
$i_V^{*}[\co_\Y]_V$ can be used to determine whether an integrally
indecomposable weight $\alpha\in \pam$ lies in $\pno$ or
$\pta$. Specifically, $\alpha$ lies in $\pno$ if and only if
$1-e^{-\gt}$ is a simple factor of
$i_V^{*}[\co_\Y]_V$. Computationally, however, an algorithm which
utilizes this idea would seem to present difficulties, since it is
often possible to express $\pull$ in many different ways.  Determining
whether $1-e^{-\gt}$ is a factor of $\pull$ is nontrivial in general.

In this section we show that when $\Y$ is the Kazhdan-Lusztig variety
$\kxw$ in an appropriate space $V\subseteq G/B$, then a particular
expression $P_{w,\vs}$ for $\pull$ due to Graham and Willems has the
property that, if we assume that $\gt$ is integrally indecomposable in
$\pam=\gP(V)$, then whenever $1-e^{-\gt}$ is a simple factor of
$\pull$, it is a factor of $P_{w,\vs}$ in a trivial fashion (see
Theorem \ref{t.sk-standard}).  Thus the expression $P_{w,\vs}$ allows
us to detect simple factors $1-e^{-\gt}$ of $\pull$, $\gt$ integrally
indecomposable, in a computationally simple manner.

We begin with two subsections reviewing properties of Kazhdan-Lusztig
and Schubert varieties in $G/B$.

\subsection{Unipotent subgroups and affine spaces in $G/B$}
\label{ss.unip}

Let $G$ be a semisimple algebraic group defined over a algebraically
closed field $k$ of characteristic 0, $B$ a Borel subgroup, $B^-$ the
opposite Borel subgroup, and $T=B\cap B^-$, a maximal torus. Let
$W=N_G(T)/T$, the Weyl group of $G$, and let $S$ be the set of simple
reflections of $W$ relative to $B$.

We consider several unipotent subgroups of $G$, referring the reader
to \cite[14.12]{Bor:91} for a more detailed discussion of their
properties. Let $U$ and $U^-$ be the unipotent radicals of $B$ and
$B^-$ respectively.  Their weights, $\gP(U)$ and $\gP(U^-)$, are by
definition the positive and negative roots $\gP^+$ and $\gP^-$
respectively. The unipotent subgroup $xU^-x^{-1}$, which we denote by
$U^-(x)$, has weights $x\gP^-$.  The unipotent subgroup $U^-(x)\cap U$
has weights $\gP(U^-(x)\cap U)=x\gP^-\cap \gP^+$, which equals
$\{\ga\in \gP^+\mid x^{-1}(\ga)<0\}$, the inversion set of
$x^{-1}$. Similarly, $U^-(x)\cap U^-$ has weights $\gP(U^-(x)\cap
U^-)=x\gP^-\cap \gP^-=\{\ga\in \gP^-\mid x^{-1}(\ga)<0\}=-\{\ga\in
\gP^+\mid x^{-1}(\ga)>0\}=-(\gP^+\setminus I(x^{-1}))$.  The unipotent
subgroups discussed in this paragraph are isomorphic to their Lie
algebras \cite[Remark 14.4]{Bor:91}.  In particular, they are
isomorphic to affine spaces, with which we often associate them.

The variety $G/B$ is called the full flag variety. The $T$-fixed
points of $G/B$ are of the form $uB$, $u\in W$.  By Bruhat
decomposition, under the mapping $\gz:U^-(x) \to G/B$, $y\mapsto
y\cdot xB$, the unipotent subgroup $U^-(x)$ embeds as a $T$-stable
affine space in $G/B$ containing $xB$.  We denote this affine space by
$C_{x}$.  The unipotent subgroup $U^-(x)\cap U$ embeds as an affine
subspace, which we denote by $V$.  We note that the weight spaces of
$C_x$ are one dimensional and that the weights lie in an open
half-space; thus the same is true of $V$.

\subsection{Schubert and Kazhdan-Lusztig varieties in $G/B$}

The Schubert variety $\xw\subseteq G/B$ is defined to be
$\overline{B^-wB}$, the Zariski closure of the $B^-$ orbit through
$wB$. The Kazhdan-Lusztig variety $\kxw$ is defined to be $V\cap \xw$.
As the following lemma shows, locally, the two varieties differ only
by an affine space with well-prescribed weights.

\begin{lemma}\label{l.sch-kl} Let $w\leq x \in W$. Then
  \begin{enumerate}
  \item $C_{x}\cong \vamkmb\times V$.
  \item $\xw \cap C_{x}\cong \vamkmb \times \kxw$.
  \item $T_x\xw \cong \vamkmb \times T_x\kxw$.
  \item $\gP(T_x\xw)=-(\gP^+\setminus I(x^{-1}))\sqcup \gP(T_x\kxw)$.
  \end{enumerate}
  The isomorphisms (i) - (iii) are $T$-equivariant isomorphisms of
  varieties.
\end{lemma}
\begin{proof}
  \noindent (i) \cite[(4.6)]{GrKr:17}.

  \noindent (ii) is an application of \cite[Lemma 4.6]{GrKr:17}, with
  $H=\vamkmb$, $\Y=\vamkb$, and $Z=\xw\cap C_x$.

  \noindent (iii) follows from (ii) and the fact that the tangent
  space of a product is isomorphic to the product of the tangent
  spaces (see \cite[Proposition 6.9]{GoWe:10}).

  \noindent (iv) We saw in Sectionn \ref{ss.unip} that
  $\gP\vamkmb=-(\gP^+\setminus I(x^{-1}))$. Combined with (iii), this
  yields the desired result.
\end{proof}

Lemma \ref{l.sch-kl}(iv) shows how to produce $\gP(T_x\xw)$ from
$\gP(T_x\kxw)$. In Section \ref{s.removable}, we saw that information
about the latter can be obtained from $\pullkk$. The next proposition
asserts that this pullback is equal to $\pullb$, for which there are
known formulas, in particular \eqref{e.gw} below.

\begin{lemma}\label{l.kthsk} Let $w\leq x\in W$. Then
  $\pullkk=\pullb$.
\end{lemma}
  \begin{proof}
    By Lemma \ref{l.sch-kl}(i) and (ii), we can apply \cite[Lemma
    2.1]{GrKr:17} to obtain $\pullkk=i_{C_x}^{*}[\co_{X^w\cap
      C_x}]_{C_x}$. Since pullbacks in equivariant $K$-theory are
    defined locally, $i_{C_x}^{*}[\co_{X^w\cap C_x}]_{C_x}=\pullb$.
  \end{proof}

  \subsection{Formulas for weights of the tangent
    space}\label{ss.normal}

Fix a reduced expression $\vs=(s_1,\ldots,s_l)$ for $x$.  The elements
of the inversion set $I(x^{-1})=\Phi^+\cap x \Phi^-$ are given
explicitly by the formula $\gg_i=s_1\cdots s_i(\ga_i)$,
$i=1,\ldots,l$, where $\ga_i$ is the simple root corresponding to
$s_i$ \cite{Hum:90}. The following result is due to Graham
\cite{Gra:02} and Willems \cite{Wil:06}:
\begin{theorem}\label{t.gw}
  Let $w\leq x\in W$, and let $\vs=(s_1,\ldots,s_l)$ be a reduced
  sequence of simple reflections for $x$. Then
  \begin{equation}\label{e.gw}
    \pullb = 
    \sum_{\vt\in \ct_{w,\vs}}(-1)^{e(\vt)}\prod_{i\in \vt}(1-e^{-\gg_i})\in R(T)
  \end{equation}
\end{theorem}

Denote the expression $\sum_{\vt\in
  \ct_{w,\vs}}(-1)^{e(\vt)}\prod_{i\in \vt}(1-e^{-\gg_i})$ by
$P_{w,\vs}$. By Lemma \ref{l.kthsk}, we have

\begin{corollary}\label{c.gw} Let $w\leq x\in W$. Then
  $\pullk=P_{w,\vs}$.
\end{corollary}

\begin{remark}
  In general, there exist numerous expressions for $\pullk$.  Lemma
  \ref{l.cl-subvar} assures us that there exists an expression as a
  polynomial in $1-e^{-\gg}$, $\gg\in \gP(V)=I(x^{-1})$; $P\ws$ is
  such an expression.
\end{remark}

We shall say that $1-e^{-\gg_j}$ is an \textbf{explicit factor} of
$P_{w,\vs}$ if $1-e^{-\gg_j}$ occurs among the factors of every
summand $\prod_{i\in \vt}(1-e^{-\gg_i})$ of $P_{w,\vs}$, or
equivalently, if $j$ belongs to every $\vt\in \ct_{w,\vs}$ (see Remark
\ref{r.explicit}).  Since all of the $\gg_j$, $j=1,\ldots,l$, are
distinct, every explicit factor of $P_{w,\vs}$ is a simple factor of
$\pullk$.  The following theorem tells us that when $\gg_j$ is
integrally indecomposable in $I(x^{-1})$, the converse is true as
well. 

\begin{theorem}\label{t.sk-standard}
  Let $w\leq x\in W$, and let $\gg_j$ be integrally indecomposable in
  $I(x^{-1})$.  If $1-e^{-\gg_j}$ is a simple factor of $\pullk$, then
  it is an explicit factor of $P_{w,\vs}$.
\end{theorem}
\begin{proof}
  Let $C$ be the coordinate ring of the tangent cone to $\kxw$ at $x$.
  We will assume that $1-e^{-\gg_j}$ is not an explicit factor of
  $P_{w,\vs}$ and show that $-\gg_j$ is a weight of $C$ (of
  multiplicity 1).  Proposition \ref{p.ind-remov} then implies that
  $1-e^{-\gg_j}$ is not a simple factor of $\pullk$, completing the
  proof.  Let $[l]$ denote $\{1,\ldots,l\}$.

  By Proposition \ref{p.phivw}, we have
  \begin{equation}\label{e.charbpar}
    \Char C = \frac{\pullk}{\lambda_{-1}(V^{*})}
    = \frac{\sum_{\vt\in \ct_{w,\vs}}(-1)^{e(\vt)}\prod_{i\in \vt} 
      (1-e^{-\gamma_i})}{\prod_{i\in [l]}(1-e^{-\gg_i})}
  \end{equation}
Each summand of \eqref{e.charbpar} can be simplified:
{\allowdisplaybreaks
  \begin{align*}
    \frac{(-1)^{e(\vt)}\prod_{i\in\vt}(1-e^{-\gamma_i})}{\prod_{i\in
        [l]}(1-e^{-\gg_i})}
    &= (-1)^{e(\vt)}\frac{1}{\prod_{i\notin\vt}(1-e^{-\gg_i})}\\
    &= (-1)^{e(\vt)}\prod_{i\notin\vt}(1 + e^{-\gg_i}+ e^{-2\gg_i}+\cdots)\\
    &= (-1)^{e(\vt)}\sum\nolimits_{\zeta\in \zcone \{\gg_i:i\notin\vt\}}
    n_{\zeta} e^{-\zeta}\\
    &= (-1)^{e(\vt)} n_{\gg_j}e^{-\gg_j}+\text{ other terms}
\end{align*}
}where $n_{\zeta}$ is the number of ways to express $\zeta$ as a
nonnegative integer linear combination of the $\gg_i$, $i\notin\vt$,
and ``other terms'' refers to an infinite linear combination of
characters with no $e^{-\gg_j}$ term.  Since $\gg_j$ is integrally indecomposable
in $I(x^{-1})$, $n_{\gg_j}=1$ if $j\notin \vt$ and $n_{\gg_j}=0$ if
$j\in \vt$. Thus
\begin{equation}\label{e.charcmon}
  \frac{(-1)^{e(\vt)}\prod_{i\in\vt}(1-e^{-\gamma_i})}{\prod_{i\in [l]}(1-e^{-\gg_i})}
  =
  \begin{cases}
    (-1)^{e(\vt)}e^{-\gg_j}+\text{ other terms}, &\text{if }
    j\notin \vt\\
    \text{other terms}, &\text{if } j\in \vt
  \end{cases}
\end{equation}

  According to \eqref{e.charbpar}, $\Char C$ is the sum of fractions
  as in \eqref{e.charcmon}, one for each $\vt\in \ct_{w,\vs}$.
  Therefore the coefficient of $e^{-\gg_j}$ in $\Char C$ is
  \begin{equation*}
    N=\sum_{\{\vt\in \ct_{w,\vs}:j\notin \vt\}} (-1)^{e(\vt)}
  \end{equation*}
  Setting $\vs_j=(s_1,\ldots,\hat{s}_j, \ldots, s_l)$, we have
  \begin{equation*}
    N=\sum_{\vt\in \ct_{w,\vs_j}} (-1)^{e(\vt)}
  \end{equation*}
  The assumption that $1-e^{-\gg_j}$ is not an explicit factor of
  $P_{w,\vs}$ assures us that $\ct_{w,\vs_j}\neq \emptyset$, and thus
  this sum equals 1 by Corollary \ref{c.euler}. Since $N\neq 0$,
  $-\gg_j$ is a weight of $C$.
\end{proof}

Denote $\gP(T_x\kxw)$ and $\gP(V/T_x\kxw)$ by $\pta$ and $\pno$
respectively, so $\pta\sqcup \pno=\gP(V)=I(x^{-1})$.

\begin{corollary}\label{c.standard}
  Let $w\leq x\in W$, and suppose that $\gg$ is an integrally
  indecomposable element of $I(x^{-1})$. If $\gg\in \pno$, then
  $1-e^{-\gg}$ is an explicit factor of $P_{w,\vs}$.
\end{corollary}
\begin{proof}
  Since $\gg\in\pno$, $1-e^{-\gg}$ is a simple factor of $\pullk$, by
  Proposition \ref{p.nor-remov}. Thus $1-e^{-\gg}$ is an explicit
  factor of $P_{w,\vs}$, by Theorem \ref{t.sk-standard}.
\end{proof}

Let $m=\ell(w)$, and define {\allowdisplaybreaks
\begin{equation*}
  \crr\ct_{w,\vs}=\{\vt=(t_1,\ldots,t_m)\subseteq [l]\mid
  s_{t_1}\cdots s_{t_m}=w\}.
\end{equation*}
Parts (i) - (iii) of the following theorem summarize the main findings
of this section thus far. Parts (iv) and (v) provide a computationally
simpler method of determining whether $\gg_j$ lies in $\pta$, by
allowing us to substitute $\crr\ct_{w,\vs}$ for $\ct_{w,\vs}$, and
thus to perform calculations in the Weyl group rather than the 0-Hecke
algebra. Part (vi) gives an alternative characterization of (v) in
terms of Demazure products.
\begin{theorem}\label{t.summary}
  Let $w\leq x\in W$, and let $\vs=(s_1,\ldots,s_l)$ be a reduced
  expression for $x$. If $\gg_j$ is an integrally indecomposable
  element of $I(x^{-1})$, then the following are equivalent:
\begin{enumerate}
\item $\gg_j\in \pta$.
\item $1-e^{-\gg_j}$ is not an explicit factor of $P_{w,\vs}$.
\item There exists $\vt\in \ct_{w,\vs}$ not containing $j$.
\item There exists $\vt\in \crr\ct_{w,\vs}$ not containing $j$.
\item There exists a reduced subexpression of
  $(s_1,\ldots,\wh{s}_j,\ldots,s_l)$ for $w$.
\item $\gd((s_1,\ldots,\wh{s}_j,\ldots,s_l))\geq w$.
\end{enumerate}
\end{theorem}
\begin{proof}
  (i) $\Leftrightarrow$ (ii) by Proposition \ref{p.ind-remov} and
  Corollary \ref{c.standard}; (ii) $\Leftrightarrow$ (iii) and (iv)
  $\Leftrightarrow$ (v) are due to definitions of $P_{w,\vs}$,
  $\ct_{w,\vs}$, and $\crr\ct_{w,\vs}$. The proof of (iii)
  $\Leftrightarrow$ (iv) follows from $\ct_{w,\vs} \supseteq
  \crr\ct_{w,\vs}$ and the fact that every element of $\ct_{w,\vs}$
  contains an element of $\crr\ct_{w,\vs}$.

  (v) $\Leftrightarrow$ (vi) There exists a reduced subexpression of
  $\vs$ for $w$ not containing $s_j$ if and only if there exists a
  subexpression of $(s_1,\ldots,\wh{s}_j,\ldots,s_l)$ for $w$ if and
  only if $\gd((s_1,\ldots,\wh{s}_j,\ldots,s_l))\geq w$, where the
  last equivalence is due to \cite[Lemma 3.4 (1)]{KnMi:04}.
\end{proof}

\begin{remark}\label{r.zixi}
  For $\gg_j\in I(x^{-1})$, it is known that in type $A$, $\gg_j\in
  \pta$ if and only if $s_1\cdots \wh{s}_j\cdots s_l\geq w$
  \cite{LaSe:84}.  Theorem \ref{t.summary} states that if $\gg_j$ is
  integrally indecomposable in $I(x^{-1})$, then $\gg_j\in\pta$ if and
  only if $\gd((s_1,\ldots,\wh{s}_j,\ldots,s_l))\geq w$. These two
  statements imply that in type $A$, if $\gg_j$ is integrally
  indecomposable in $I(x^{-1})$, then
  $\gd((s_1,\ldots,\wh{s}_j,\ldots,s_l))\geq w$ if and only if
  $s_1\cdots \wh{s}_j\cdots s_l\geq w$.  That this holds for all
  $w\leq x$ would seem to imply that
  $\gd((s_1,\ldots,\wh{s}_j,\ldots,s_l))=s_1\cdots \wh{s}_j\cdots
  s_l$. This is indeed true, and the above argument can be made
  rigorous. In \cite{GrKr:21}, it is shown that the statement extends
  to all simply-laced types. It is also shown that if $\gg_j$ is {\em
    rationally indecomposable} in $I(x^{-1})$, then
  $\gd((s_1,\ldots,\wh{s}_j,\ldots,s_l))=s_1\cdots \wh{s}_j\cdots s_l$
  in all types.
\end{remark}

\begin{remark}\label{r.gnot-int-ind}
  Suppose that $\gg_j$ is not integrally indecomposable in
  $I(x^{-1})$.  Then statements (ii) -- (vi) of Theorem
  \ref{t.summary} are still equivalent, but statement (i) is no longer
  equivalent to the other five in general.  The following example
  shows that (vi) $\Rightarrow$ (i) can fail.  In type $A_2$, let
  $\vs=(\gs_1,\gs_2,\gs_1)$, where $\gs_i$ is the simple transposition
  which exchanges $i$ and $i+1$. Let $w=\gs_1$ and $j=2$. Then
  $\gd((\gs_1,\wh{\gs}_2,\gs_1))=\gd((\gs_1,\gs_1))=\gs_1\geq w$, so
  (vi) holds. However, $\gs_1\wh{\gs}_2\gs_1=e\not\geq w$. Thus
  $\gg_2\notin \pta$ (see Remark \ref{r.zixi}), so (i) fails.

  We note that $\gg_j$ is required to be integrally indecomposable in
  $I(x^{-1})$ for our proofs of both implications of Theorem
  \ref{t.summary} (i) $\Leftrightarrow$ (ii).
\end{remark}

\section{Partial flag varieties and cominuscule elements} \label{s.partial}
Let $P$ be a parabolic subgroup containing $B$. In Section \ref{ss.extending-gp} we
show that Lemma \ref{l.sch-kl} and Theorem \ref{t.summary} extend from
$G/B$ to $G/P$ with no changes other than notation.  In Section \ref{ss.comin}
we apply the results to cominuscule elements of $W$ and cominuscule $G/P$.

\subsection{Extending results to $G/P$}\label{ss.extending-gp}
Let $P$ be a parabolic subgroup containing $B$.
Let $L$ be the Levi subgroup of $P$ containing $T$, and $W_P=N_L(T)/T$, the Weyl
group of $L$. Each coset $uW_P$ in $W/W_P$ contains a unique
representative of minimal length; denote the set of minimal length
coset representatives by $W^P\subseteq W$. Unless stated otherwise, in
this subsection we assume that all Weyl group elements lie in $W^P$.
The $T$-fixed points of $G/P$ are of the form $uP$, $u\in W^P$.

Let $P^-$ the opposite parabolic subgroup to $P$, and let $U_P^-$ be
the the unipotent radical of $P^-$.  Under the mapping $\gz:U_P^-(x)
\to G/P$, $y\mapsto y\cdot xP$, the unipotent subgroup $U_P^-(x)$
embeds as a $T$-stable affine space in $G/P$ containing $xP$.  The
unipotent subgroup $U_P^-(x)\cap U$ embeds as an affine subspace, which
we denote by $V_P$.

The Schubert variety $\xwp\subseteq G/P$ is defined to be
$\overline{B^-wP}$, the Zariski closure of the $B^-$ orbit through
$wP$. The Kazhdan-Lusztig variety $\kxwp$ is defined to be $V_P\cap
\xwp$.

The following result appears in \cite[Section 7.3]{Knu:09}:

\begin{theorem}\label{t.yizi} Let $w\leq x\in W^P$. Then $V_P\cong V$ and 
  $\kxwp\cong \kxw$.
\end{theorem}

The next theorem extends the main results of Section \ref{s.particular} to $G/P$.

\begin{theorem}\label{t.gpsummary}
  Let $w\leq x\in W^P$. 
  \begin{enumerate}
    \item $\gP(T_x\xwp)=\gP(U_P^-(x)\cap U^-)  \sqcup
    \gP(T_x\kxwp)$.
  \item Let $\gg_j$ be an integrally indecomposable element of
    $I(x^{-1})$. Then $\gg_j \in \gP(T_x\kxwp)$ if and only if
    $\gd((s_1,\ldots,\hat{s}_j,\ldots,s_l))\geq w$.
  \end{enumerate}
\end{theorem}
\begin{proof}
  (i) Lemma \ref{l.sch-kl}(i)-(iii) remain valid if all quantities are
  replaced by their analogs in $G/P$. In particular, $T_x\xwp\cong
  \vamkm \times T_x\kxwp$.

\noindent (ii) Since $\kxwp \cong \kxw$, $T_x\kxwp \cong T_x\kxw$. Thus all
  parts of Theorem \ref{t.summary} remain valid if $\pta=\gP(T_x\kxw)$
  in Theorem \ref{t.summary}(i) is replaced by $\gP(T_x\kxwp)$.
\end{proof}

\subsection{Application to cominuscule Weyl group elements and
  cominuscule $G/P$} \label{ss.comin}
In this subsection we discuss conditions on $x$ under which all
elements of $I(x^{-1})$ are integrally indecomposable, and thus, for
any Kazhdan-Lusztig variety containing $x$, Theorems \ref{t.summary}
and \ref{t.gpsummary}(ii) recover all weights of the tangent space at
$x$.  In particular, we show that our results completely describe the tangent
spaces of Schubert varieties in cominuscule $G/P$.

\begin{definition}
  The element $x\in W$ is said to be {\em cominuscule} if there exists
  $v\in \ft$ such that $\ga(v)=-1$ for all $\ga\in I(x^{-1})$.
\end{definition}

This notion was introduced and studied by Peterson (see \cite[Section
5.2]{GrKr:17} or \cite{Ste:01} for discussion). In type $A$, the
cominuscule Weyl group elements are precisely the 321-avoiding
permutations \cite[p. 25]{Knu:09}.  As noted in \cite{GrKr:17}, the equality
$I(x) = -x^{-1} I(x^{-1})$ implies that $x$ is cominuscule if and only if $x^{-1}$ is.

\begin{proposition}\label{p.comin-w}
  If $x\in W$ is cominuscule, then all elements of $I(x^{-1})$ are
  integrally indecomposable.
\end{proposition}
\begin{proof}
  If $x$ is cominuscule, then there exists $v\in \ft$ such that
  $\ga(v)=-1$ for all $\ga\in I(x^{-1})$. Assume that some $\gb\in
  I(x^{-1})$ is integrally decomposable. Then $\gb=\sum_{i=1}^m\gb_i$,
  where $m\geq 2$, $\gb_i\in I(x^{-1})$. Since $\gb(v)=-1$ and
  $\gb_i(v)=-1$ for all $i$, this leads to a contradiction.
\end{proof}

\begin{remark} \label{r.noncomin}
The converse of the above proposition is false: 
there exist non-cominuscule elements $x$ such that every element of $I(x^{-1})$ is
integrally indecomposable.  The following example is a variation and extension
of \cite[Remark~5.4]{Ste:01}.  In type $D_4$, with the conventions of \cite{Hum:90}, consider the element
$x = s_2 s_1 s_3 s_4 s_2$.   The inversion set $I(x^{-1})$ is equal to
$\{ \gre_1 -\gre_3, \gre_1+\gre_2, \gre_2 - \gre_3, \gre_2 - \gre_4, \gre_2 + \gre_4 \}$.
Every element of $I(x^{-1})$ is integrally indecomposable, but the element 
$x$ is not cominuscule (cf.~\cite[Remark~5.4]{Ste:01}).
Note that \cite{Ste:01} uses a different numbering of the vertices of the Dynkin diagram in which node $3$ has degree
$3$ (see \cite[Remark 2.7]{Ste:01}), so he writes the element $x$ as $s_3 s_1 s_2 s_4 s_3$.
\end{remark}

\begin{definition}
  The maximal parabolic subgroup $P\supseteq B$ is said to be {\em
    cominuscule} if the simple root $\ga_i$ corresponding to $P$
  occurs with coefficient 1 when the highest root of $G$ is written as
  a linear combination of the simple roots.
\end{definition}

If $P$ is cominuscule, then the corresponding flag variety $G/P$ is
said to be cominuscule as well.  We refer the reader to \cite[Chapter
9]{BiLa:00}, \cite[Chapter VI, \S 1, Exercise 24]{Bou:02},
\cite{GrKr:15} for more on cominuscule $G/P$. The following
proposition gives an important class of cominuscule Weyl group
elements.

\begin{proposition}\label{p.comin-p}
  If $x\in W^P$, where $P$ is cominuscule, then $x$ is a cominuscule
  element of $W$.
\end{proposition}
\begin{proof}
 If $x\in W^P$, then $U^-(x) \cap U = U_P^-(x) \cap U$ (see the discussion before Lemma 4.1 in
 \cite{GrKr:15}, cf.~\cite{Knu:09}).  Hence
 $$
I(x^{-1}) = \Phi((U^-(x) \cap U) =\Phi( U_P^-(x) \cap U) \subset x \Phi(U_P^-).
$$
 Let $\ga_1,\ldots,\ga_r$ denote the simple roots of $G$; these form
a basis for $\ft^{*}$. Denote the dual basis of $\ft$ by
$\xi_1,\ldots,\xi_r$. Assume that $P$ corresponds to the simple root
$\ga_i$. Since $P$ is cominuscule, \cite[Lemma 2.8]{GrKr:15} implies $\ga_i$ must occur with
coefficient $-1$ in all $\ga\in \Phi(U_P^-)$ (when $\ga$ is written as
a linear combination of the simple roots), so for all such $\ga$, we have
$\ga(\xi_i) = -1$.  It follows that $v=x\xi_i$ satisfies $\ga(v)=-1$ for all
$\ga \in  x \Phi(U_P^-)$.  Hence $\ga(v) = -1$ for all $\ga \in I(x^{-1})$, so
$x$ is a cominuscule element of $W$.
\end{proof}

\begin{remark} \label{r.summary}
The results of this subsection imply that if $P$ is cominuscule and $w\leq x\in W^P$,
then Theorem
\ref{t.gpsummary} characterizes all weights of $T_x\kxwp$ and $T_x\xwp$.  More generally,
suppose $x \in W$ is any cominuscule element (or more generally any element such that each element of
$I(x^{-1})$ is integrally indecomposable).  Then Lemma \ref{l.sch-kl} and Theorem \ref{t.summary} characterize all weights of
$T_x\kxw$ and $T_x X^w$.
If in addition $P \supset B$ is a parabolic subgroup such that $w, x\in W^P$, then Theorem
\ref{t.gpsummary} characterizes all weights of $T_x\kxwp$ and $T_x\xwp$.
\end{remark}


\bibliographystyle{amsalpha}
\bibliography{intremovableroots}

\end{document}